\documentclass[11pt,reqno]{amsart}

\usepackage{geometry}
\usepackage{fullpage}
\usepackage{calrsfs}
\usepackage[OT2,T1]{fontenc}
\usepackage[english]{babel}
\usepackage{times}

\usepackage{etoolbox}
\patchcmd{\section}{\scshape}{\bfseries}{}{}
\makeatletter
\renewcommand{\@secnumfont}{\bfseries}
\makeatother

\newtheorem{introtheorem}{Theorem}

\newtheorem{introcorollary}[introtheorem]{Corollary}
\newtheorem{introprop}[introtheorem]{Proposition}
\theoremstyle{definition}
\newtheorem{introexample}[introtheorem]{Example}
\newtheorem{introex}[introtheorem]{Exercise}
\newtheorem{introdef}[introtheorem]{Definition}
\newtheorem{introremark}[introtheorem]{Remark}
\theoremstyle{plain}
\newtheorem{theorem}{Theorem}[subsection]
\newtheorem{proposition}[theorem]{Proposition}
\newtheorem{lemma}[theorem]{Lemma}

\theoremstyle{definition}
\newtheorem{definition}[theorem]{Definition}

\newtheorem{remark}[theorem]{Remark}

\usepackage{amsfonts}
\usepackage{amssymb}
\usepackage{hyperref}
\usepackage{mathtools}

\newcommand{\Z}{\mathbf{Z}}
\newcommand{\ddet}{\mathrm{ddet}}

\newcommand{\F}{\mathbf{F}}
\newcommand{\Fp}{\F_{\hspace*{-0.15em}p}}
\newcommand{\Fq}{\F_{\hspace*{-0.15em}q}}
\newcommand{\FpM}{\F_{\hspace*{-0.15em}p^M}}
\newcommand{\FpN}{\F_{\hspace*{-0.15em}p^N}}

\newcommand{\A}{\mathbf{A}}
\newcommand{\M}{\mathrm{M}}
\DeclareMathOperator{\End}{\mathrm{End}}

\DeclareMathOperator{\tr}{tr}
\newcommand{\Fix}{\mathrm{Fix}}
\newcommand{\Ga}{\mathbf{G}_a}
\newcommand{\Tr}{\iota}
\renewcommand{\geq}{\geqslant}
\renewcommand{\leq}{\leqslant}

\usepackage[dvipsnames]{xcolor}
\newcommand{\mynote}[2]{\noindent
{\bfseries\sffamily\scriptsize#1}
{\small$\blacktriangleright$\textsf{\textsl{#2}}$\blacktriangleleft$}}
\newcommand\JB[1]{\mynote{JB}{{\color[rgb]{0.9,0.9,0} #1}}}

\begin{document}

\date{\today}
\title{Multiband linear cellular automata \\ and endomorphisms of algebraic vector groups}
\author[J.~Byszewski]{Jakub Byszewski}
\address{\normalfont Wydzia\l{} Matematyki i Informatyki Uniwersytetu Jagiello\'nskiego, ul.\ S.\ \L ojasiewicza 6, 30-348 Krak\'ow, Polska}
\email{jakub.byszewski@uj.edu.pl}
\author[G.~Cornelissen]{Gunther Cornelissen}
\address{\normalfont Mathematisch Instituut, Universiteit Utrecht, Postbus 80.010, 3508 TA Utrecht, Nederland}
\email{g.cornelissen@uu.nl}
\thanks{JB was supported by the National Science Centre, Poland under grant no.\ 2018/29/B/ST1/01340}


\begin{abstract} \noindent We propose a correspondence between certain multiband linear cellular automata---models of computation widely used in the description of physical phenomena---and endomorphisms of certain algebraic unipotent groups over finite fields. The correspondence is based on the construction of a universal element specialising to a normal generator for any finite field. \\ We use this correspondence to deduce new results concerning the temporal dynamics of such automata, using our prior, purely algebraic, study of the endomorphism ring of vector groups. These produce `for free' a formula for the number of fixed points of the $n$-iterate in terms of the $p$-adic valuation of $n$, a dichotomy for the Artin--Mazur dynamical zeta function, and an asymptotic formula for the number of periodic orbits.  \\ Since multiband linear cellular automata simulate higher order linear automata (in which states depend on finitely many prior temporal states, not just the direct predecessor), the results apply equally well to that class. 
\end{abstract}

\maketitle
 

\section*{Introduction}

This paper proposes a bridge between (small parts of) the world of computational physics and the world of algebraic geometry in positive characteristic. 

\subsection*{Context} Cellular automata are concurrent models of computation described in the middle of the 20th century by Ulam and von Neumann \cite{vN}, and are helpful in modelling various biological, chemical, and physical processes, ranging from pigment distribution on seashells to the Belousov–Zhabotinsky reaction, 
which forms the foundation of the construction of a `chemical computer' \cite{ChemComp}. From the perspective of mathematics, computer science and statistical physics, even one-dimensional cellular automata show complex behaviour, as pointed out by Wolfram and others \cite{Wolfram};  the very simple looking rules of certain one-dimensional cellular automata lead to  computational universality in the sense of being able to simulate any Turing machine \cite{Cook}. 

Here, we are concerned with a very simple class of linear one-dimensional cellular automata for which the states are described by elements of an $r$-dimensional vector space 
$V:=\Fp^r$ over a finite field $\Fp$ with $p$ elements, as studied in \cite{Lieven}. Note that by working over $\Fp$, `linear' is the same as `additive', and additive automata over $\FpN^r$ can be seen as linear automata over $\Fp^{Nr}$. Linearity is a restriction: for example, the resulting machines are never universal  \cite[Theorem 4.12]{Lieven}; however, in general, linear cellular automata are transitive dynamical systems with a dense set of periodic points, so chaotic in the sense of Devaney \cite{Favati}. 

An early detailed study of dynamical properties of some examples can be found in \cite{Martin}. A very general framework of group-based automata can be found in \cite[Chapter 8]{Coornaert}, where invertibility and the Garden-of-Eden theorem are discussed (see also \cite[Proposition 3.2]{Lieven}). For $r=1$, Ward connected such automata to number-theoretical dynamical systems, and used this to provide very clean proofs of some properties related to periodic points and entropy \cite{Ward-entropy} (see also \cite[Chapter 9]{BCH}). By contrast, for general $r$ several `bits' (units of information, in our case elements\footnote{Sometimes these are called `pits' for $\Fp$ and `qits' for $\Fq$, but we will not use this terminology.} of $\Fp$) are stored in one state; the corresponding systems are called multiband linear cellular automata \cite{Kari}, and the dynamics of such automata is considered much more complex \cite[p.\ 64]{Grin}.\footnote{We will refrain from referring to $r$ as `dimension', since that word is used in  the theory of automata for the grid size dimension, which in this paper is always $1$; nevertheless, via our correspondence $r$ becomes the dimension of the associated vector group.}  The interest in such multiband automata arises from the fact that they can be used to simulate higher-order linear cellular automata, i.e.\ automata where transition rules depend not only on the immediately preceding configuration, but on finitely many earlier configurations (\cite[Proposition 3.1]{Lieven}, \cite{Den}). In fact, most (bio-)chemical applications use state spaces of greater complexity than merely a single bit. 

In this paper we link the dynamics of such automata to that of endomorphisms of the vector group $\Ga^r(\overline \Fp)$, where $\overline \Fp$ is the algebraic closure of $\Fp$. Results on periodic points for the cellular automaton follow in a straightforward manner from our previous study of dynamics on algebraic groups \cite{BCH}, similarly to how Ward was able to use prior number-theoretical work for $r=1$. Our results apply to general $r \geq 1$ (and thus, also to higher order additive cellular automata over $\Fp$) and allow us to avoid considering special cases, only requiring some algebra as in \cite[\S 5.2]{BCH}. 

\begin{introex} \label{ex1}  
Let $f_n$ denote the number of solutions $x \in \overline \Fp$ of the equation $f^{\circ n}(x)=x$ where $f(x)=x^p+x$, using the notation $f^{\circ n} = f \circ \dots \circ f$ for the $n$-fold iterate of $f$, and let $g_n$ denote the number of sequences $y=(y_i) \in \Fp^{\Z}$ for which $g^{\circ n}(y)=y$ with $g((y_i)) = (y_i+y_{i+1})$;  then $f_n=g_n$ for all $n$.\footnote{\raisebox{0.5em}{\rotatebox{180}{In fact,  $\log_p f_n = n-p^{v_p(n)}$ where $v_p$ is the $p$-adic valuation.}}}
\end{introex}
 
Is the equality between the number of periodic points of given period for the two maps $f$ and $g$ in Exercise \ref{ex1} a coincidence? The map $f$ is an endomorphism of 
 the additive group $\Ga(\overline \Fp)$ and the map $g$ describes a linear cellular automaton with alphabet $\Fp$. As we will see below, there is a general connection between these two concepts that explains the numerical coincidence.

\subsection*{Definitions and main result} Let us specify the definition of the cellular automata we will consider. 
\begin{introdef} Let $S=V^{\Z}$ be the set of bi-infinite sequences $(y_i)=(y_i)_{i \in \Z}$ of elements of the vector space $V:=\Fp^r$, carrying the product topology induced by the discrete topology on $V$. A ($V$-)\emph{multiband linear cellular automaton} is a continuous $\Fp$-linear self-map $g \colon S \rightarrow S$ that commutes with the shift $s \colon S \rightarrow S$ given by $s((y_i))=(y_{i+1})$. We say the automaton is \emph{on the alphabet $\Fp^r$}. 
 We will denote the cellular automaton by the map $g$, or, if we want to stress the space, also by $(S,g)$. 
\end{introdef}

In the terminology of topological dynamics, $g$ is an endomorphism of the full shift on $S$. By the Curtis--Hedlund--Lyndon theorem (\cite{Hedlund}, \cite[Theorem 1.8.1]{Coornaert}) and additivity, $g$ is determined by its \emph{local rule} $$g((y_i))_0 = \sum_{j \in J} m_j y_{j},$$ where $J$ is a finite set (a `window') and $m_j \in \mathrm{End}(V) \cong \M_r(\Fp)$ are matrices. We can consider $g$ as a \emph{convolutional encoder} \cite[\S 1.6]{Marcus} in the following sense: we represent the sequence $(y_i)$ by the bi-infinite power series $Y(Z):=\sum y_i Z^i \in V[\![Z,Z^{-1}]\!]$, and $g$ by multiplication with the Laurent polynomial $$G(Z):=\sum_{j \in J} m_j Z^j \in \mathrm{End}(V)[Z,Z^{-1}].$$ We can rewrite $G(Z)$ as a matrix of Laurent polynomials $G(Z) \in \M_{r \times r}(\Fp[Z,Z^{-1}])$. We will call the cellular automaton \emph{one-sided} if $G(Z) \in \M_{r \times r}(\Fp[Z])$, that is, if it is a matrix of polynomials. 

\begin{introremark} The endomorphism $g$ defines a discrete-time dynamical system on $S$: if we denote by $y^{(t)}=(y^{(t)}_i) \in S$ a state `at time $t \in \{0,1,2,\dots\}$', then we set $y^{(t+1)}=g(y^{(t)})$. The dynamical properties of this system are called `temporal', in order to distinguish them from properties of the entries of $y$ at a fixed time, e.g.\ with respect to the shift map, which are called `spatial'. A state is only dependent on the immediate predecessor in time, but as we remarked above, by allowing $r>1$, we can simulate any higher order linear cellular automaton, i.e.\ an additive automaton in which a state  $y^{(t)}$ depends on finitely many predecessors $y^{(t-1)},\dots,y^{(t-k)}$ in time.  If this dependency is written as $$Y^{(t)}(Z) = \sum_{j=1}^s G_j(Z) Y^{(t-j)}(Z)$$ for $Y^{(t)}(Z) = \sum y_i^{(t)}Z^i$ with $G_j(Z) \in \M_{r \times r}(\Fp[Z,Z^{-1}])$, then the first-order simulation corresponds to the block matrix 
$$ {\tiny \left( \begin{array}{ccccc} G_1(Z) & G_2(Z) &  \ldots & & G_s(Z) \\ I_r & & & &  \\ & I_r & & & \\ & & \tiny{\ddots} & &   \\ & & & I_r &  \end{array} \right)} \in \M_{rs \times rs}(\Fp[Z,Z^{-1}])  \mbox{ acting on } {\tiny \left( \begin{array}{c} Y^{(t-1)}(Z) \\ Y^{(t-2)}(Z)\\   {\vdots} \\Y^{(t-s)}(Z) \end{array} \right)} \in V^s[\![Z,Z^{-1}]\!].$$
\end{introremark}

The other objects we consider are the following. 
\begin{introdef} We let $\Ga$ denote the additive algebraic group, so for a field $K$, $\Ga(K)$ is the additive group of $K$. An algebraic group of the form $\Ga^r$ is called a \emph{vector group} (a particular example of a unipotent algebraic group); its $K$-points form an $r$-dimensional $K$-vector space, and we denote them as \emph{column} vectors. If $K=\overline \Fp$, let $F$ denote the Frobenius operator of raising to the power $p$ (applied to elements of $\overline \Fp$ or to all coordinates of an element of $\Ga^r(\overline \Fp)$). 
 We say that an endomorphism of $\Ga^r$ is \emph{defined over $\Fp$} if it commutes with the Frobenius operator $F\colon \Ga^r(\overline \Fp)\to \Ga^r(\overline \Fp)$.  
 \end{introdef} 
 
 Endomorphisms $\sigma \in \End_{\overline \Fp}(\Ga^r)$ of the vector group $\Ga^r$ over $\overline \Fp$ can be studied using noncommutative algebra: they are described by $r\times r$ matrices $m_\sigma$ with entries from the skew polynomial ring $\overline \Fp\langle \phi \rangle$, where $\phi a = a^p \phi$ for all $a \in \overline \Fp$ \cite[\S 5.2]{BCH}. The endomorphisms defined over $\Fp$ form a subring isomorphic to the matrix ring over the (commutative) polynomial ring $\F_p[\phi]$. This allows us to identify one-sided automata directly with endomorphisms of algebraic groups by identifying the variables $\phi$ and $Z$. 
 
We now describe a less obvious result that gives a direct connection between fixed point sets for multiband linear cellular automata and endomorphisms of vector groups. 
 
 \begin{introdef} Suppose $f \colon X \rightarrow X$ is a self-map on some set $X$. 
 \begin{itemize} \item  We define the \emph{set of fixed points} of $f$ as $\Fix(f)=\{x \in X \colon f(x)=x\}.$ A fixed point of $f^{\circ n}$ for some $n \geq 1$ is called a \emph{periodic point}. 
\item We call the map $f$ \emph{confined} if $\# \Fix(f^{\circ n})$ is finite for all $n \geq 1$. 
\item We denote by $\mathcal P \subseteq S$ the subset of sequences that are periodic for the shift map $s \colon S \rightarrow S$, and call these \emph{shift-periodic sequences}.  
  \end{itemize}
    \end{introdef} 
  Given a cellular automaton $g$, the finite orbit of a periodic point of $g \colon S \rightarrow S$ is sometimes called a `temporal cycle' to distinguish them from cycles for the shift map $s \colon S \rightarrow S$, which are called `spatial cycles'.

\begin{introtheorem} \label{main}
There is an injective group homomorphism $\Tr \colon \Ga^r(\overline \Fp) \rightarrow (\Fp^r)^{\Z}$ such that the following hold:
\begin{enumerate}\item\label{m1} We have $\Tr(Fx)=s(\Tr(x))$, where $F$ is the Frobenius operator, and $s$ the shift. 
\item  \label{m2} The image of $\Tr$ consists exactly of the subgroup $\mathcal P$ of shift-periodic sequences of $(\Fp^r)^{\Z}.$
\item \label{m3} There is a bijection $\sigma \mapsto g_\sigma$ between endomorphisms $\sigma$ of the group $\Ga^r$ defined over $\Fp$ and one-sided $\F^r_p$-multiband linear cellular automata $g_\sigma$ such that for all $n \geq 1$, the map $\iota$ induces a group isomorphism 
\begin{equation} \label{matchp} \Fix(\sigma^{\circ n}) \stackrel{\iota}{\cong} \Fix(g_\sigma^{\circ n}) \cap \mathcal P.  \end{equation} 
\item \label{m4} The endomorphism $\sigma$ is confined if and only if the cellular automaton $g_\sigma$ is confined, in which case all periodic points of $g_\sigma$ are shift-periodic and $\iota$ induces a group isomorphism \begin{equation} \label{match} \Fix(\sigma^{\circ n}) \stackrel{\iota}{\cong} \Fix(g_\sigma^{\circ n}).  \end{equation} 
\end{enumerate}
\end{introtheorem} 

The main innovation in the theorem seems to be the link between, on the one hand, points over \emph{the algebraic closure} of $\Fp$ with Galois action and, on the other hand, periodic sequences \emph{with entries in} $\Fp^r$ with the action of the shift map. Being defined over $\Fp$ is natural from the point of view of the correspondence: it means invariance under $F$, which translates to the shift-commutation property of the cellular automaton. The action of the absolute Galois group $\mathrm{Gal}(\overline \Fp/\Fp)$ on the algebraic group corresponds to profinite shifts in the cellular space, cf.\ Remark \ref{defprofshift} below.  Confinedness is required to match sets of fixed points exactly; without this condition, for example, the identity map fixes the countable set $\Ga^r(\overline \Fp)$ on the one side and the uncountable set $S$ on the other side, and the $\overline \Fp$-rational points in the algebraic group correspond to the shift-periodic sequences only.  

\subsection*{Corollaries of the correspondence} Since the group $\Ga^r$ is commutative, the number of fixed points of an endomorphism $\sigma \in \End(\Ga^r)$ in $\Ga^r(\overline \Fp)$ is equal to the cardinality of $\ker(\sigma-1)(\overline \Fp)$, which can be computed as the quotient of the degree of the map $\sigma-1$ by its inseparability degree, and this can be identified with $$p^{\deg_\phi(\ddet(m_{\sigma-1}))-v_\phi(\ddet(m_{\sigma-1}))}$$ where $\ddet$ is the Dieudonn\'e determinant, and $\deg_\phi$ and $v_\phi$ are the degrees of the leading and trailing coefficients in $\phi$ \cite[Lemma 5.2.2]{BCH}. If $\sigma$ is defined over $\Fp$, one may use linear algebra of matrices over (commutative) polynomial rings and the usual determinant rather than the Dieudonn\'e determinant.

We have used these observations in \cite{BCH} to understand the orbit structure of such systems. The results can be directly translated to the setting of \emph{one-sided} cellular automata; to the best of our knowledge, in this generality the results are new. To obtain analogous results for general two-sided automata, we first perform some shifts and reduce to a counting problem for more general coincidence sets rather than periodic points (that are coincidence sets for iterates of a map and the identity map), to which the methods from \cite{BCH} can also be applied.

\begin{introprop} \label{cor1} 
If $g$ is a confined linear multiband cellular automaton on an alphabet $\Fp^r$, then for all $n \geq 1$ we have \begin{equation} \label{fix} \log_p \# \Fix(g^{\circ n}) = na-t_n p^{v_p(n)} \end{equation} for some $a \in \Z_{\geq 0}$ and some periodic sequence $(t_n)$ with $t_n \in \Z_{ \geq 0}$ of period $\varpi$ coprime to $p$ and satisfying $t_n=t_{\mathrm{gcd}(n,\varpi)}$, where $v_p$ is the $p$-adic valuation. The exact value of $a,t_n,\varpi$ is as given in Remark \ref{format}. 
\end{introprop}

Following Artin and Mazur \cite{AM}, the sequence counting fixed points of iterates of a confined map can be studied using the \emph{dynamical zeta function} $$\zeta_g(z):=\exp \sum_{n \geq 1} \left( \# \Fix(g^{\circ n}) z^n/n \right).$$  By transferring our previous results concerning this zeta function for endomorphisms of additive groups to the setting of cellular automata, we obtain the following new results about cellular automata. (The case $a=0$ needs some prior attention, see Lemma \ref{cora=0}.)  

\begin{introcorollary} \label{cor3} 
Let $g$ be a confined multiband linear cellular automaton on an alphabet $\Fp^r$ with invariants $a$ and $(t_n)$ as in Proposition \ref{cor1}. Then the dynamical zeta function of $g$ is either the rational function $1/(1-p^az)$ or it cannot be analytically continued over any point of the circle of convergence $|z|=1/p^a$. The first case happens if and only if $t_n=0$ for all $n$.
\end{introcorollary}

\begin{introcorollary} \label{cor2} 
For a confined linear multiband cellular automaton $g$ on an alphabet $\Fp^r$, the number $P_\ell$ of periodic orbits of length $\ell$ is asymptotically equal to 
\begin{equation} \label{PNT} {p^{\ell a - t_{\ell} p^{v_p(\ell)}}}/{\ell}+ O(\sqrt{p^{\ell a}}) \end{equation}  
as $\ell \rightarrow +\infty$, where $a$ and $t_\ell$ are as in Proposition \ref{cor1}. 
\end{introcorollary}

\begin{introexample} For $r=1$ and $g=F$, Corollary \ref{cor2} is the so-called Prime Polynomial Theorem first mentioned in \S 342--347 of the drafted eighth chapter `Disquisitiones generales de congruentiis' of Gau{\ss}'s Disquisitiones Arithmeticae, cf.\ \cite[Theorem 2.2]{Rosen}, stating that the number of irreducible polynomials of degree $\ell$ over $\Fp$ is asymptotically equal to $p^\ell/\ell+ O(p^{\ell/2})$ as $\ell \rightarrow +\infty$. The dynamical zeta function in this case is the innocuous Weil zeta function of $\A^1_{\Fp}$, given by $1/(1-pz)$. 
This setup corresponds to the cellular automaton with state space $\Fp^{\Z}$ and map $(y_i) \mapsto (y_{i+1})$. 
\end{introexample} 

\begin{introremark}
Equation \eqref{PNT} can be used to study the limit behaviour of the \emph{periodic orbit counting function} defined as  $$\pi(X) = \sum_{\ell \leq X} P_\ell.$$ If $a=0$, there are finitely many periodic points (cf.\ Lemma \ref{cora=0}) and $\pi(X)$ is eventually constant. If $a\neq 0$, then the function $X\pi(X)/p^{aX}$ has a well-defined limit precisely if $t_n=0$ for all $n$, and that limit then equals $p^a/(p^a-1)$. In general, the function  $X \pi(X)/p^{aX}$ has uncountably many accumulation points arising from restricting the values of $X$ to sequences of integers converging simultaneously modulo the period $\varpi$ and $p$-adically. See \cite[\S 12.4 \& 12.5]{BCH} for details. 
\end{introremark}

 \section{Construction and first properties of the map $\iota$} 
 
 \subsection{Traces} 
 
Consider an extension of finite fields ${\FpN}/{\FpM}$ with $M$ dividing $N$. Let $G_{N,M}$ denote its Galois group---cyclic of order $N/M$---, and write $$\tr_{N,M}(x):=\sum\limits_{\tau \in G_{N,M}} \hspace*{-0.5em} \tau(x)$$ for the relative trace map.  This is an additive surjective map  $\Ga(\FpN)  \rightarrow \Ga(\FpM)$ that commutes with the action of the Frobenius operator $F$. Furthermore, the bilinear map $\FpN \times \FpN \rightarrow \FpM$ given by $$\langle x ,y \rangle := \tr_{N,M}(xy)$$ is non-degenerate \cite[\S 2.1.4, Theorem 2.1.84]{HFF}. 

\begin{definition} 
Taking inverse limits in the category of abelian groups with the relative trace maps as connecting homomorphisms, define $$\mathcal F = \varprojlim\limits_{\stackrel{M \mid N}{\tr_{N,M}}} \Ga(\FpN).$$  
\end{definition} 

If $\{G_i\}$ is an inverse system of groups with $G=\varprojlim G_i$, the inverse limit of group rings $\varprojlim \Fp[G_i]$ is denoted by $\Fp[\![G]\!]$ and called the \emph{complete  group ring}. This profinite ring contains the group ring $\Fp[G]$ as a dense subset \cite[\S 5.3]{Ribes}.

 \begin{proposition} 
 $\mathcal F$ is a free $\Fp[\![\widehat \Z]\!]$-module of rank $1$. Thus, there exists $\alpha \in \mathcal F$ that for every $N$ projects to an element $\alpha_N \in \FpN$ that generates a normal basis.  
 \end{proposition} 
 
 \begin{proof} 
The existence of a normal basis for a finite extension of finite fields \cite[Theorem 5.2.1]{HFF} shows that $\Ga(\FpN)$ is a free $\Fp[G_{N,1}]$-module of rank $1$; a generator of this module is precisely a normal basis generator of $\FpN$ over $\Fp$. 

Let $B_N \subset \Ga(\FpN)$ denote the set of all generators of $\Ga(\FpN)$ as an $\Fp[G_{N,1}]$-module. The map $\tr_{N,M}$ is surjective, and so via this map $\Ga(\FpM)$ can be regarded as a quotient module of $\Ga(\FpN)$ (where $\Ga(\FpM)$ is regarded as a  $\Fp[G_{N,1}]$-module via the surjective map of rings $\Fp[G_{N,1}]\to \F_p[G_{M,1}]$). Thus, for any $\alpha_N \in B_N$ and $M{\mid}N$, $\tr_{N,M}(\alpha_N)$ lies in $B_M$, and so $B_N$ form an inverse system of sets. 

Since the inverse limit of a system of nonempty finite sets is nonempty \cite[3.2.13]{Engelking}, there is an element $\alpha \in \mathcal F$ that maps to some $\alpha_N \in B_N$ for all $N$. Using these compatible generators $\alpha_N$, we construct an isomorphism $$\mathcal F \cong \varprojlim \Fp[G_{N,1}] = 
\Fp[\![\varprojlim G_{N,1}]\!].$$ The space $\mathcal F$ is indeed an $\Fp[\![\widehat \Z]\!]$-module of rank $1$ under the identification of $\widehat \Z$ with the absolute Galois group $\mathrm{Gal}(\overline \Fp/\Fp) = \varprojlim\limits_{N} G_{N,1}$ of $\Fp$.\qedhere
\end{proof} 

\begin{remark} 
In defining $\mathcal F$, we use an inverse limit over trace maps to arrive at an uncountable limit space, which turns out to be a free rank-$1$ module under the action of the complete modular group ring of the absolute Galois group of $\Fp$. By contrast, the direct limit of these spaces under inclusions is $\Ga(\overline \Fp)$, which is a countable abelian group.  We are using $\mathcal F$ as a receptacle for the existence of a compatible set of normal bases for a tower of additive groups of finite fields. 
Finding compatible models of the entire field structure (i.e.\ finding the structure constants of the operation of multiplication on compatible vector space bases) is much more difficult, see \cite[\S 11.7]{HFF}. 
\end{remark}

\subsection{Definition of the map $\iota$} Choose an element $\alpha = \varprojlim \alpha_N \in \mathcal F$ with $\alpha_N \in \FpN$ satisfying $$\tr_{N,M}(\alpha_N)=\alpha_M$$ for all $M,N$ with $M \mid N$. Let $y \in \overline \Fp$. Compatibility with the relative trace maps implies that the expression $\tr_{N,1}(\alpha_N y)$ for any $N$ with $y \in \FpN$ is well-defined and independent of the choice of $N$.  We denote it simply by $\tr(\alpha y).$ 
 
\begin{definition} 
 The map 
 $\Tr_\alpha \colon \Ga^r(\overline \Fp) \rightarrow (\Fp^r)^{\Z}$ is given by $$\Tr_\alpha(x):= (\tr(\alpha x^{p^j}))_{j \in \Z},$$ where we apply the $p$-th power operation and trace map entry-wise to column vectors.  
 \end{definition} 
 
 \begin{proposition} \label{mainprop} 
 The map $\Tr_\alpha$ has the following properties: 
 \begin{enumerate}
 \item \label{1} $\Tr_\alpha$ is a group homomorphism for addition.  
    \item \label{3} $\Tr_\alpha$ is injective. 
 \item \label{2} The image of $\Tr_\alpha$ coincides with the subgroup $\mathcal P$ of all sequences that are periodic for the shift operator in $(\Fp^r)^{\Z}$.
   \item \label{4} If  $F$ denotes the Frobenius operator and $s$ is the shift operation, then $\Tr(Fx)=s(\Tr(x))$.  \end{enumerate} 
  \end{proposition} 
  
  \begin{proof}
  Part \eqref{1} is clear. For \eqref{3}, suppose $\tr(\alpha x^{p^j})=0$ for some $x \in \FpN$ and all $j$. Then, by invariance of the trace under Frobenius, also $\tr(\alpha^{p^{-j}} x)=0$. Thus $x$ is orthogonal to the set of generators $\alpha_N^{p^{-j}}$ of $\FpN$ as a vector space over $\Fp$, and by the nondegeneracy of the trace form, $x=0$. 
  
  For part \eqref{2}, suppose that $x \in \Ga^r(\overline \Fp)$ belongs to $\Ga^r(\FpN)$. Now $\Tr_\alpha$ induces an injective map from $\Ga^r(\FpN)$ to the set of periodic sequences in $(\Fp^r)^{\Z}$ of period $N$, and this map is also surjective for cardinality reasons. 

  For \eqref{4}, note that $\Tr(F(x)) = (\tr(\alpha F (x^{p^j}))) = (\tr(\alpha x^{p^{j+1}}))=s(\tr(\alpha x^{p^j}))$. 
 \end{proof} 
 
\begin{remark} As the space $(\Fp^r)^{\Z}$ is compact, we like to think of these properties as saying that the shift space over the elementary abelian group of order $p^r$ is a `dynamical compactification' of the points of the affine algebraic group $\Ga^r(\overline \Fp)$ (which correspond to the periodic sequences). 
\end{remark} 

We let $\iota$ denote the map $\iota_\alpha$ for any choice of $\alpha$ as above. The dynamical results will not depend on the many possible choices of $\alpha$. 

\begin{remark} \label{defprofshift} 
The \emph{profinite shift} by $\psi \in \widehat{\Z}$ is the map $s_{\psi} \colon \mathcal P \rightarrow \mathcal P$ given by $$s_{\psi}((y_j))= (y_{j+\psi \bmod{N}})$$ for $(y_j)$ having period $N$. Under our correspondence and the identification of $\widehat{\Z}$ with $\mathrm{Gal}(\overline \Fp/\Fp)$, profinite shifts on periodic sequences correspond to the action of the corresponding element of the absolute Galois group of $\Fp$ on the algebraic group $\Ga^r(\overline \Fp)$.  \end{remark}

 \section{Remaining proofs} 

\subsection{Results concerning fixed points and periodic orbits} 

We start with a lemma about general subshifts of $S=(\Fp^r)^{\Z}$\JB{,} which is not directly related to cellular automata. We use the standard terminology of calling a finite string of symbols from the alphabet $\Fp^r$ a \emph{word}.  A subset $Y$ of $S$ is a \emph{subshift of finite type} if it is defined by a finite list of `excluded' words, in the sense that $Y$ comprises the sequences in $S$ that do not contain any of those excluded words as subwords. 

\begin{lemma} \label{pseudo} Suppose $Y \subseteq S$ is a subshift of finite type such that $Y$ is a subgroup and $Y \cap \mathcal P$ is finite. Then $Y$ is contained in $\mathcal P$  and thus $Y$ is finite. 
\end{lemma} 

\begin{proof} 
Suppose that $Y$ contains a non-shift-periodic sequence $y$. For any $M \geq 0$, $Y$ also contains such a sequence $y'$ with a block $0^M$ of $M$ consecutive zeros; indeed, some word of length $M$ appears twice at different positions in $y$, and by shifting so that these become aligned and then subtracting, the block $0^M$ is found in a sequence $y'$ of the form $y'=y-s^{\circ k} y$ for some appropriate $k$. Note that $y' \in Y$ since $Y$ is a shift-invariant subgroup. We will specify an appropriate value for $M$ later in the proof. 

Let $s$ denote an integer strictly larger than both the size of the largest block of consecutive zeros in the non-zero periodic sequences in $Y$ (which is finite by the assumed finiteness of $Y \cap \mathcal P$), as well as the length of the longest excluded word defining $Y$ (which is finite by the finite-type assumption). 

Since $y$ is not periodic, $y' \neq 0$. Hence either to the left or to the right of the block $0^M$ in $y'$ there will be a non-zero symbol $a$. We assume that such a symbol $a$ occurs to the right (if it occurs to the left, the argument applies mutatis mutandis). Then there is a word of length $s$ that will occur twice at different positions in $y'$ to the right of the position of the chosen $a$, where the two occurrences of that word are at distance at most $p^{rs}$ (by the pigeon-hole principle).  
We again form a sequence $y''=y'-s^{\circ \ell} y' \in Y$ for some $\ell \leq p^{rs}$ that aligns the two occurrences. 

In $y''$, the original block of zeros $0^M$ might be shortened to $0^{M'}$, but in any case, $M' \geq M-p^{rs}$. Assuming $M=p^{rs}+s$, we have $M' \geq s$ and the constructed sequence $y''$ has two blocks of $s$ consecutive zeros. Furthermore, if $b$ is the word occurring in between these two blocks, there is at least one non-zero entry in $b$ since the non-zero symbol $a$ occurs between $0^M$ and the chosen words in $y'$. 

Finally, we construct a bi-infinite sequence $y''' = \dots 0^s b 0^s b 0^s \dots$ by shift-periodically repeating the word $0^s b$. Since $s$ is larger than the maximal length of excluded words defining $Y$ and all words of smaller length occurring in $y'''$ are not excluded since they also occur in $y''$, we have that $y''' \in Y$.  Thus, $y'''$ is a non-zero periodic sequence in $Y$ with a block of more than the maximal permitted number of consecutive zeros (by the choice of $s$), and this is a contradiction. We conclude that $Y\subseteq \mathcal P$. 
\end{proof} 

\begin{proof}[Proof of Theorem \ref{main}] 
The injectivity of $\iota$ and statements \eqref{m1} \and \eqref{m2} were proven in Proposition \ref{mainprop}. For the remaining parts, we first construct the bijection $\sigma \mapsto g_\sigma$. 

Let $\sigma \in \End(\Ga^r)$ be an endomorphism defined over $\Fp$. It is described by a matrix $$m_\sigma = m_0 +m_1 \phi + \cdots + m_N \phi^N \in \M_r(\Fp[\phi])$$ with $m_i \in \M_r(\Fp)$. 
The corresponding cellular automaton $g_\sigma$ on the space $S=(\F^r_p)^{\Z}$ is defined by the rule $$(y_i) \mapsto (m_0 y_i + m_1 y_{i+1} + \cdots + m_s y_{i+N}).$$ Then $\sigma \mapsto g_\sigma$ is a bijection between endomorphisms of $\Ga^r$ defined over $\Fp$ and one-sided linear cellular automata on $S$ (note that it is clear how to construct the inverse map). 

We now consider the matching of fixed points. Suppose $x=(x_1,\dots,x_r)^\intercal \in \Ga^r(\overline \Fp)$ is a fixed point of $m_{\sigma}$, so $m_{\sigma}(x)=x$, i.e.\ 
$$ m_0 x + m_1 x^p + \dots + m_N x^{p^N} = x $$
with the operation of raising to the $p$-th power applied cooordinatewise to $x$. Recall that matrices $m_i$ have entries in $\Fp$, so if we raise this equation to the $p^j$-th power, we find that 
$$ m_0 x^{p^j} + m_1 x^{p^{j+1}} + \dots + m_N x^{p^{j+N}} = x^{p^j} $$
for all $j$. 
Multiplying this equation by $\alpha$ and taking traces, we find
\begin{equation} \label{says} m_0 \tr(\alpha x^{p^j}) + m_1 \tr(\alpha x^{p^{j+1}}) + \dots + m_N  \tr(\alpha x^{p^{j+N}}) = \tr(\alpha x^{p^j}). \end{equation} 
Set $y_j:= \tr(\alpha x^{p^j})$ and recall that \begin{equation} \label{excl} g_\sigma((y_j)) = m_0 y_j + m_1 y_{j+1}+ \dots + m_N  y_{j+N}, \end{equation} Equation \eqref{says} says that 
$ g_\sigma((y_j)) = (y_j), $ so $(y_j)$ is a fixed point of $g_\sigma$. At the same time, by Proposition \ref{mainprop}\eqref{2}, $\iota(x)=y \in \mathcal P$. This shows that 
$\iota(\Fix(\sigma)) \subseteq \Fix(g_\sigma) \cap \mathcal P$. 

For the inclusion in the other direction,  if $g_\sigma(y)=y$ and $y \in \mathcal P$, then $y = \iota(x)$ for some $x \in \Ga^r(\overline \Fp)$ and, reversing the above argument, we see that $\sigma(x)=x$. We can apply the argument to $\sigma^{\circ n}$ for any $n$, and this proves Equation \eqref{matchp}. 

Assume that $g_\sigma$ is confined, meaning that $\Fix(g_\sigma^{\circ n})$ is finite for all $n$. Since $g_\sigma$ commutes with the shift $s$, $\Fix(g_\sigma^{\circ n})$ is a finite shift-invariant subgroup of $S$, and hence consists of shift-periodic sequences. Thus,  $\Fix(g_\sigma^{\circ n}) \subseteq \mathcal P$, and Equation \eqref{match} follows. 

Finally, we show that confinedness of $\sigma$ implies that of $g_\sigma$. Note that $Y:=\Fix(g_\sigma^{\circ n})$ is a shift-invariant subgroup of $S$ with $Y \cap \mathcal P$ finite by Equation \eqref{matchp}.  Furthermore, $Y$ is a subshift of finite type; indeed, let $N$ be as in the local rule \eqref{excl} for $g_\sigma$ and exclude all words of length $\leq nN+1$ that do not map to zero under the local rule for $g_\sigma^{\circ n}-\mathrm{id}$. By Lemma \ref{pseudo}, it follows that $Y$ is finite, which means that $g_{\sigma}$ is confined.
\end{proof} 

In the next proof, we will use the following concept. 

\begin{definition} 
The \emph{coincidence set} $\mathrm{Coin}(f,g)$ of two self-maps $f \colon X \rightarrow X$ and $g \colon X \rightarrow X$ on a set $X$ is defined as $$\mathrm{Coin}(f,g)=\{ x \in X \colon f(x)=g(x) \}.$$ 
\end{definition} 

\begin{proof}[Proof of Proposition \ref{cor1}]
For one-sided automata, by Theorem \ref{main}, we are reduced to proving the statement for confined $\sigma \in \End(\Ga^r)$, and in that case the formula for the number of fixed points is given in \cite[Theorem 5.2.17]{BCH}. In the general case, we can reason as follows: assume that the local rule is given by $g((y_i))_0 = \sum_{k=M}^N m_k y_{k}$ for some $M,N \in \Z$ and some matrices $m_k \in \M_r(\Fp)$. Without loss of generality, we can assume $M<0$ (possibly allowing $m_k=0$ for some $k$). Consider $\tilde g = gs^{-M}$; then $$\Fix(g^{\circ n}) =  \mathrm{Coin}(\tilde g^{\circ n}, s^{\circ -Mn}).$$ Note that this set is finite (by the assumption that $g$ is confined) and shift-periodic (by the shift-commuting property of $g$); it follows that the set consists of shift-periodic sequences only.
Therefore, letting $\tilde \sigma \in \End(\Ga^r)$ be such that $g_{\tilde \sigma} = \tilde g$ and using that $s$ corresponds to $F$ on the algebraic group side, we find, by the same argument used to prove Theorem \ref{main}\eqref{m4}, that 
$$ \mathrm{Coin}(\tilde g^{\circ n}, s^{\circ -Mn}) \stackrel{\iota}{=}  \mathrm{Coin}(\tilde \sigma^{\circ n}, F^{\circ -Mn}). $$ Thus, we can compute 
\begin{align} \label{fixct} 
\log_p \#\Fix(g^{\circ n}) & = 
\log_p \# \mathrm{Coin}(\tilde g^{\circ n}, s^{\circ -Mn}) \nonumber \\ & = \log_p \#   \mathrm{Coin}(\tilde \sigma^{\circ n}, F^{\circ -Mn}) \nonumber \\  &  =  \log_p \# \ker ({\tilde \sigma}^{\circ n} - F^{\circ -Mn} | \Ga^r(\overline \Fp)) \nonumber \\ & \stackrel{(\dagger)}{=}   \deg_{\phi} (\det(m_{\tilde \sigma}^n - \phi^{-Mn})) - v_{\phi} (\det(m_{\tilde \sigma}^n - \phi^{-Mn})) \nonumber \\
& = -Mnr + \deg_{\phi}(\det((m_{\tilde \sigma}\phi^M)^n-1)) + Mnr - v_{\phi} (\det((m_{\tilde \sigma} \phi^M)^n - 1)). 
 \end{align}
Equality $(\dagger)$ holds for the following reason: (a) for an endomorphism $\tau$ of a commutative algebraic group with finite kernel (equivalently, a surjective endomorphism), $\# \ker(\tau)$ is the quotient of the degree of $\tau$ by the inseparability degree of $\tau$ \cite[Formula (4.1)]{BCH}; (b) the map $\tau = {\tilde \sigma}^{\circ n} - F^{\circ -Mn} \in \End(\Ga^r)$ has finite kernel, since by the above formula its cardinality equals $\#\Fix(g^{\circ n})$ and we assumed $g$ to be confined; (c) we can apply the formula expressing the degree and inseparability degree of an endomorphism of a vector group in terms of the valuation of the determinant given in \cite[Lemma 5.2.2]{BCH} (notice that since we work over $\Fp$, the Dieudonn\'e determinants in that reference can be replaced by usual determinants). 

We now apply \cite[Lemma 5.2.5]{BCH} to the valuation $v_{\phi^{-1}}=-\deg_{\phi}$ and $v_{\phi}$ and for matrices over the field of rational functions $\Fp(\phi)$. For $v$ either one of these valuations, we find that 
$$v(\det((m_{\tilde \sigma} \phi^M)^n - 1)) = - a_v n + t_n^v p^{v_p(n)}$$ for some $a_v \in \Z_{\geq 0}$ and periodic sequence $(t^v_n)$ with $t^v_n \in \Z_{ \geq 0}$ having period $\varpi_v$ coprime to $p$ and satisfying $t^v_n=t^v_{\mathrm{gcd}(n,\varpi_v)}$. Equation \eqref{fix} follows by adding together  these formul\ae.  
\end{proof} 

\begin{remark} \label{format} 
From the proof of \cite[Lemma 5.2.5]{BCH} emerge the following exact formulas for $a, t_n$ and $\varpi$. Choose some extensions $\tilde{v}_{\phi^{-1}}$ and $\tilde{v}_{\phi}$ of the valuations $v_{\phi^{-1}}=-\deg_{\phi}$ and $v_{\phi}$ to a fixed algebraic closure $\overline{\Fp(\phi)}$, and let $\lambda_1,\ldots, \lambda_r\in \overline{\Fp(\phi)}$ denote the eigenvalues of $m_{\tilde \sigma} \phi^M$, taken with multiplicities.  Note that, by the assumed confinedness of $g$, none of the $\lambda_i$ is a root of unity. 
The formula for $a$ is \begin{equation} \label{forma} a=\sum_i \max(-\tilde{v}_{\phi^{-1}}(\lambda_i),0) + \max(-\tilde{v}_{\phi}(\lambda_i),0).\end{equation}
Given a valuation $v$ on $\overline{\Fp(\phi)}$ and an element $\lambda \in \overline{\Fp(\phi)}$, let $$m^v(\lambda):= \left\{ \begin{array}{ll} \min \{ m \in \Z_{>0} \colon v(\lambda^m-1)>0 \} & \mbox{ if it exists, } \\ 
\infty & \mbox{ otherwise.} 
\end{array} \right.
$$ The formula for $t_n$ is 
\begin{equation} \label{formtn}
t_n=\sum_{v \in \{\tilde v_\phi,\tilde v_{\phi^{-1}}\}}\, \sum_{m^v(\lambda_i) \mid n} v(\lambda_i^{m^v(\lambda_i)}-1),  
\end{equation}
with the convention that $\infty{\mid}n$ holds for no integer $n$. 
Finally, the period $\varpi$ is 
\begin{equation} \label{formvarpi} 
\varpi = \mathrm{lcm} \{ m^v(\lambda_i) \colon m^v(\lambda_i) \neq \infty, v \in \{\tilde v_\phi,\tilde v_{\phi^{-1}}\}, i=1,\dots,r \}. 
\end{equation}
\end{remark}

The next lemma describes when $a=0$, a case that will deserve special attention in subsequent proofs.

\begin{lemma} \label{cora=0}
Let $g$ be a confined multiband linear cellular automaton on an alphabet $\Fp^r$ with invariants $a$ and $(t_n)$ as in Proposition \ref{cor1}. The following conditions are equivalent:\begin{enumerate}
	\item $a=0$;
	\item the zero sequence is the unique periodic point of the automaton;
	\item some power $g^{\circ n}$ of the map $g$ is the zero map.  
\end{enumerate}
\end{lemma}

\begin{proof} 
From Equation \eqref{forma}, we see that $a=0$ if and only if $\tilde{v}_{\phi^{-1}}(\lambda_i)\geq 0$ and $\tilde{v}_{\phi}(\lambda_i)\geq 0$ for every $i$. Since $\{\lambda_i\}$ is invariant under the action of the absolute Galois group, this implies that $\tilde{v}(\lambda_i)\geq 0$ for every valuation $\tilde{v}$ of $\overline{\Fp(\phi)}$ lying above $v_{\phi^{-1}}$ and $v_{\phi}$. Since $\lambda_i$ are zeros of a monic polynomial over $\Fp[\phi,\phi^{-1}]$, we automatically have $\tilde{v}(\lambda_i)\geq 0$ for any valuation lying above the remaining valuations $v\neq v_{\phi^{-1}}, v_{\phi}$ of $\Fp[\phi]$. By the product formula, whenever $\lambda_i\neq 0$, the sum of all valuations of $\lambda_i$ is zero. 

Thus, $a=0$ if and only if for all $i$ either $\lambda_i=0$ or $\tilde{v}(\lambda_i)=0$ for all valuations $\tilde{v}$ of $\overline{\Fp(\phi)}$. The latter condition simply means that $\lambda_i$ lies in $\overline{\F_p}$. However, $\lambda_i$ cannot take values in $\overline{\Fp}^*$ by confinedness. Hence, $a=0$ if and only if $\lambda_i=0$ for all $i$, which is equivalent to the matrix $m_{\tilde \sigma} \phi^M$ (or simply $m_{\tilde \sigma}$) being nilpotent. Thus, $a=0$ implies that some power of $g$ is the zero map. 

Conversely, if some power of $g$ is the zero map, the constant zero sequence is the only periodic point, and so $\# \Fix(g^{\circ n}) = 1$ for all $n$, which implies that $a=0$.
\end{proof}

\begin{proof}[Proof of Corollaries \ref{cor3} and  \ref{cor2}]
This follows from Lemma \ref{cora=0} if $a=0$, so let $a \neq 0$. 
For the dichotomy for the zeta function in Corollary \ref{cor3}, see \cite[Theorem 13.2.1]{BCH} (note that our system is hyperbolic with `entropy' $a \log p$ in the sense of that reference). 
For Corollary \ref{cor2} concerning the orbit length formula, see \cite[Theorem 12.2.1(i)]{BCH}; this applies in the general (not necessarily one-sided) case because the formula \eqref{fix} holds in that generality too. 
\end{proof}

\section{Open problems} 

We finish the paper with a few open problems and questions related to extending the `bridge' between the theory of algebraic groups and cellular automata. 

\begin{enumerate}
\item A vector group is a particular kind of unipotent group. If $G$ is a general connected unipotent algebraic group and $\sigma \in \End(G)$ is a surjective algebraic group endomorphism, there exists a finite filtration of $G$ by characteristic (hence $\sigma$-invariant) algebraic subgroups such that all quotients are isomorphic to vector groups $\Ga^r$ (of possibly different dimensions $r$), and $\sigma$  induces an endomorphism of those \cite[Chapter 8]{BCH}. Furthermore, the number of fixed points of $\sigma$ on $G$ (if finite) is the product of that on the different quotients \cite[Proposition 4.8.1]{BCH}. Thus, the study of dynamics of endomorphisms $\sigma$ of unipotent algebraic groups (which can be assumed to be surjective by replacing $G$ by $\sigma^n G$ for suitable $n$) is reduced to the case where the group is a vector group. For a general unipotent group $G$ with an endomorphism $\sigma$ defined over $\Fp$, is there a generalisation of additive cellular automata that are filtered by structures whose quotients are linear cellular automata, corresponding to the filtration of $G$ with vector group quotients stable by $\sigma$? 
\item Can one generalise the main theorem from vector groups to more general algebraic groups? For which algebraic groups $G$ is there a map $\iota \colon G(\overline \Fp) \rightarrow G(\Fp)^{\Z}$ with similar properties? It seems best to specialise first to the case where $G$ is commutative or where $G$ is unipotent. In other examples, such as  $G=\mathrm{SL}_2$ for $p=2$, any group homomorphism $\iota \colon \mathrm{SL}_2(\overline{\F_{\hspace*{-0.15em}2}}) \rightarrow \mathrm{SL}_2(\F_{\hspace*{-0.15em}2})^{\Z}$ is trivial (this is because for $N\geq 2$ the group $\mathrm{SL}_2(\F_{\hspace*{-0.15em}2^N}) \cong \mathrm{PSL}_2(\F_{\hspace*{-0.15em}2^N})$ is simple, and so any homomorphism $\mathrm{SL}_2(\F_{\hspace*{-0.15em}2^N}) \to \mathrm{SL}_2(\F_{\hspace*{-0.15em}2}) \cong S_3$ is trivial).  
\item What happens if $\overline \Fp$ is replaced by $\overline K$ for an arbitrary field $K$ of characteristic $p>0$ with $r>1$? The results and proofs from \cite{BCH} do not apply anymore. In this case, it might be interesting to construct an appropriate theory for cellular automata and then apply it to the algebraic group side, reversing the philosophy of this paper.
\end{enumerate}

\bibliographystyle{amsplain}

\end{document}